\documentclass[a4paper]{article}

\usepackage{graphicx}

\usepackage{hyperref}

\usepackage{amsfonts, amsmath, amsthm}

\usepackage[margin=3cm]{geometry}

\usepackage{color}

\newtheorem{thm}{Theorem}[section]

\newtheorem{lem}[thm]{Lemma}

\DeclareMathOperator{\Tr}{Tr}

\begin{document}

\title{An explicit formula for the A-polynomial of twist knots}

\author{Daniel V. Mathews}


\maketitle

\begin{abstract}
We extend Hoste--Shanahan's calculations for the A-polynomial of twist knots, to give an explicit formula.
\end{abstract}


\section{Introduction}

Since Cooper--Culler--Gillet--Long--Shalen introduced the A-polynomial in 1994 \cite{CCGLS}, A-polynomials have found important applications to hyperbolic geometry, the topology of knot complements, and K-theory. More recently, they appear in relation to physics, in particular in the A-J conjecture.

However, calculations of A-polynomials remain relatively difficult. In particular, there are very few infinite families of knots for which A-polynomials are known. In his 1996 thesis, Shanahan \cite{Shanahan96} gave a formula for A-polynomials of torus knots. In 2004, Hoste--Shanahan \cite{Hoste-Shanahan04} gave a recursive formula for the A-polynomial of twist knots, and Tamura--Yokota \cite{Tamura-Yokota04} gave a recursive formula for the A-polynomials of $(-2,3,1+2n)$-pretzel knots. In 2011, Garoufalidis--Mattman  \cite{Garoufalidis-Mattman11} showed that the A-polynomials of $(-2,3,3+2n)$-pretzel knots satisfy a linear recursion relation, effectively demonstrating a recursive formula. Most recently, Petersen \cite{Petersen_a-polynomials} gave a description of the A-polynomials of a family of two-bridge knots $J(k,l)$ including the twist knots (illustrated below) as the resultant of two recursively-defined polynomials, and in the case of twist knots recovers the recursive formula of Hoste--Shanahan. To our knowledge this exhausts the current state of knowledge on formulas for A-polynomials of infinite families of knots.

In this short note we give an explicit, non-recursive formula for the twist knots. Let $J(k,l)$ be the family of knots illustrated in figure \ref{fig:twist_knot}; the twist knots are obtained when $k = \pm 2$. Note that $J(-k,-l)$ is the mirror image of $J(k,l)$, so its A-polynomial is obtained by replacing $M$ with $M^{-1}$ (and normalising appropriately). Further, $J(2,2n+1) = J(-2,2n)$. So it is sufficient to consider the knots $J(2,2n)$; we write $K_n$ for $J(2,2n)$.

Let $A_n (L,M)$ be the A-polynomial of $K_n$.

\begin{figure}
\begin{center}
\def\svgwidth{100pt}
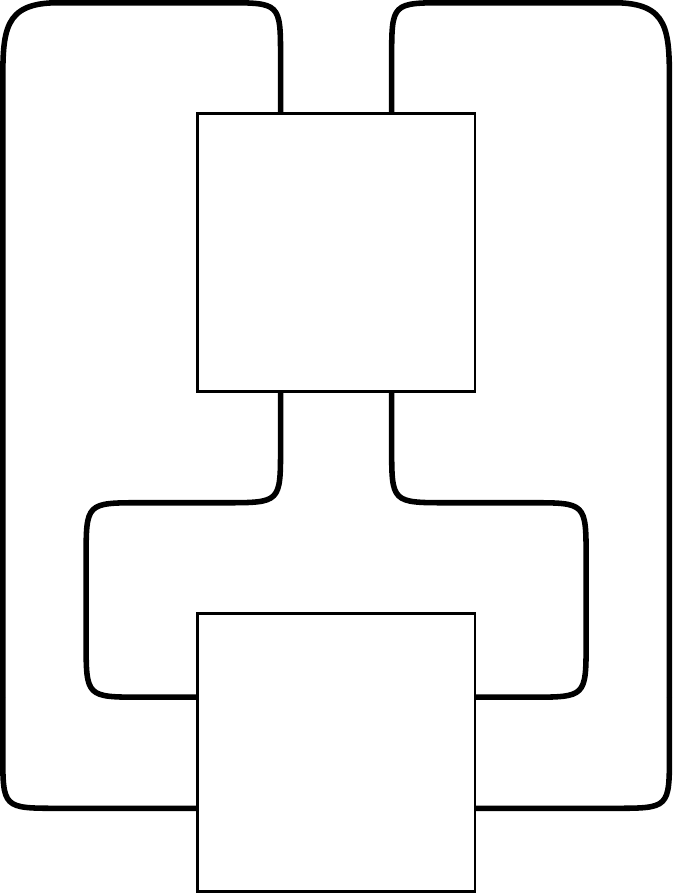
\quad \quad
\def\svgwidth{180pt}
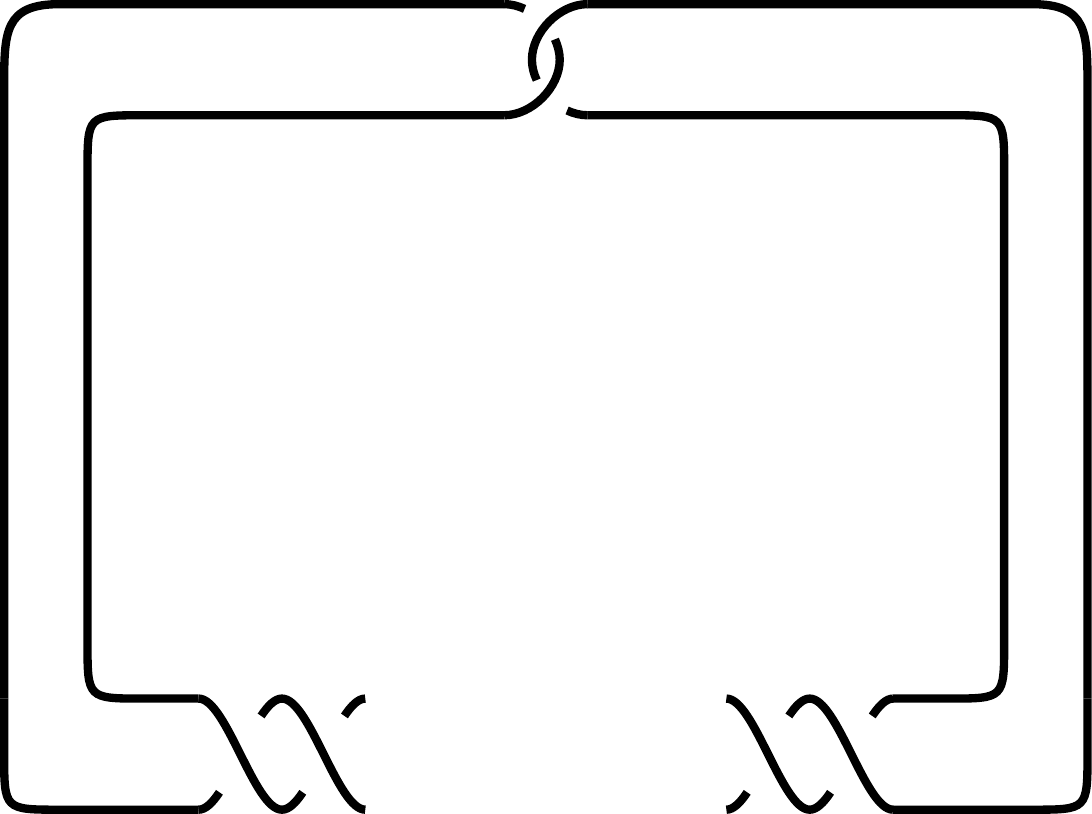
\caption{The knot $J(k,l)$, at left, is given by drawing $k$ and $l$ right-handed half twists in the boxes as shown. The twist knot $K_n$, shown right, with $n$ full twists, is equal to $J(2,2n)$.}
\label{fig:twist_knot}
\end{center}
\end{figure}

\begin{thm}
When $n \geq 0$, we have
\[
A_n (L,M) = M^{2n} (L+M^2)^{2n-1} \sum_{i=0}^{2n-1} \binom{n + \lfloor \frac{i-1}{2} \rfloor }{i} \left( \frac{ M^2 - 1 }{ L+M^2} \right)^i \left( 1 - L \right)^{ \lfloor \frac{i}{2} \rfloor } \left( M^2 - L M^{-2} \right)^{ \lfloor \frac{i+1}{2} \rfloor }.
\]
When $n \leq 0$, we have
\[
A_n (L,M) = M^{-2n} (L+M^2)^{-2n} \sum_{i=0}^{-2n} \binom{-n + \lfloor \frac{i}{2} \rfloor}{i} \left( \frac{M^2 - 1}{L+M^2} \right)^i (1-L)^{\lfloor \frac{i}{2} \rfloor} (M^2 - LM^{-2})^{ \lfloor \frac{i+1}{2} \rfloor }.
\]
\end{thm}


The proof is very direct and based on the methods of Hoste--Shanahan \cite{Hoste-Shanahan04}.

\section{Proof of theorem}

We follow Hoste--Shanahan's notation for convenience and refer there for further details. The relevant fundamental group is
\[
\pi_1 (S^3 \backslash K_n )) = \langle a,b \; | \; a (a b^{-1} a^{-1} b)^n = (a b^{-1} a^{-1} b)^n b \rangle = \langle a, b \; | \; a w^n = w^n b \rangle
\]
where $w = a b^{-1} a^{-1} b$. Both $a,b$ are meridians. A general irreducible representation $\rho: \pi_1 (S^3 \backslash J(2,2n))$ may be conjugated to be of the form
\[
\rho(a) = \begin{pmatrix} M & 1 \\ 0 & M^{-1} \end{pmatrix}, \quad
\rho(b) = \begin{pmatrix} M & 0 \\ Z & M^{-1} \end{pmatrix},
\]
where $M, Z$ are both nonzero. (Our $Z$ is $-z$ in \cite{Hoste-Shanahan04}.) The equation $\rho(aw^n ) = \rho(w^n b)$ gives four polynomial relations in $M$ and $Z$, which as discussed in \cite{Hoste-Shanahan04} reduces to a single one $r_n = 0$. Writing
\[
\rho(w^n) = \begin{pmatrix} w_{11}^n & w_{12}^n \\ w_{21}^n & w_{22}^n \end{pmatrix}
\quad \text{we have} \quad 
r_n = (M - M^{-1} ) w_{12}^n + w_{22}^n .
\]

We compute
\[
\rho(w) = \rho(a b^{-1} a^{-1} b) = 
\begin{pmatrix} w_{11}^1 & w_{12}^1 \\ w_{21}^1 & w_{22}^1 \end{pmatrix}
=
\begin{pmatrix} M^2 Z + (1-Z)^2 & M - M^{-1} + ZM^{-1} \\
-ZM^{-1} + ZM + Z^2 M^{-1} & 1 + ZM^{-2} \end{pmatrix}
\]
so that, by the Cayley-Hamilton identity (noting the above matrix has determinant $1$)
\[
\rho(w^n) = \chi \rho(w^{n-1}) - \rho(w^{n-2})
\quad \text{where} \quad
\chi = \Tr \rho(w) = Z^2 + (M - M^{-1})^2 Z + 2.
\]
Hence each entry $w_{ij}^n$ satisfies $w_{ij}^n = \chi w_{ij}^{n-1} - w_{ij}^{n-2}$. As $r_n = (M - M^{-1}) w_{12}^n + w_{22}^n$, we also have a recurrence relation
\begin{equation}
\label{eqn:recurrence}
r_n = \chi r_{n-1} - r_{n-2}.
\end{equation}

On the other hand, a longitude is given by $\lambda = w^n \overline{w}^n$, where $\overline{w} = ba^{-1} b^{-1} a$. We have
\[
\rho(\lambda) = \begin{pmatrix} L & * \\ 0 & L^{-1} \end{pmatrix}
\]
where $L = w_{11}^n \overline{w}_{22}^n + Z w_{12}^n \overline{w}_{12}^n$. Here $\overline{w}_{ij}^n$ is obtained from $w_{ij}^n$ by replacing $M$ with $M^{-1}$. We then have the relation $s_n = 0$, where
\[
s_n = w_{12}^n L + \overline{w}_{12}^n.
\]
Note $r_n$ is a polynomial satisfied by $M$ and $z$, while $s_n$ is a polynomial satisfied by $L, M$ and $z$. Eliminating $z$ from $r_n = 0$ and $s_n = 0$ gives the A-polynomial of $J(2,2n)$.

In the case of twist knots we may simplify $s_n = 0$ to the relation $s'_n = 0$ where
\[
s'_n = w_{12}^1 L + \overline{w}_{12}^1 = (M - M^{-1} + ZM^{-1}) L + M^{-1} - M + ZM.
\]
Thus $s'_n = 0$ is equivalent to 
\begin{equation}
\label{eqn:Z}
Z = \frac{(M - M^{-1})(1-L)}{M + LM^{-1}}.
\end{equation}

All of the above is in \cite{Hoste-Shanahan04}. Our strategy is simply to find an explicit formula for $r_n$ in terms of $M,Z$, and then substitute $Z$ for the expression above in terms of $L$ and $M$.

\begin{lem}
\begin{align}
\label{eqn:r_n}
r_n &= 
\sum_{i=0}^{2n-1} \binom{n + \left\lfloor \frac{i-1}{2} \right\rfloor}{i} Z^i \left( 1 + Z^{-1} (M - M^{-1})^2 \right)^{\left\lfloor \frac{i+1}{2} \right\rfloor} \quad \text{when} \quad n \geq 0 \\
&= \sum_{i=0}^{-2n} \binom{-n + \left\lfloor \frac{i}{2} \right\rfloor}{i} (-Z)^i \left( 1 + Z^{-1} (M-M^{-1})^2 \right)^{\left\lfloor \frac{i+1}{2} \right\rfloor} \quad \text{when} \quad n \leq 0
\end{align}
\end{lem}

\begin{proof}
Write $f_n$ for the claimed formula above; we show $f_n = r_n$. We give the proof for $n \geq 0$; for $n \leq 0$ the method is similar. Note that the range $0 \leq i \leq 2n-1$ is precisely the range of integers for which $0 \leq i \leq n + \lfloor \frac{i-1}{2} \rfloor$, so we can regard the sum as an infinite one, with all undefined binomial coefficients as zero.

We compute $r_0, r_1$ directly. As $\rho(w^0)$ is the identity, $r_0 = (M - M^{-1}) w_{12}^0 + w_{22}^0 = 1 = f_0$. Noting the computation of $w_{ij}^1$ above, we have $r_1 = (M - M^{-1}) w_{12}^1 + w_{22}^1 = (M - M^{-1}) (M - M^{-1} + ZM^{-1}) + 1 + ZM^{-2} = 1+Z+(M-M^{-1})^2 = f_1$. For convenience write $U = (M-M^{-1})^2$, so $\chi = Z^2 + UZ + 2 = (1+Z^{-1}U)Z^2 + 2$. We now show that $f_n$ satisfies the recurrence (\ref{eqn:recurrence}). 
\begin{align*}
\chi f_{n-1} - f_{n-2} &=
\left( (1+Z^{-1} U) Z^2 + 2 \right) \sum_i \binom{n-1+ \left\lfloor \frac{i-1}{2} \right\rfloor}{i} Z^i \left( 1 + Z^{-1} U \right)^{\left\lfloor \frac{i+1}{2} \right\rfloor} \\
&\quad \quad \quad
- \sum_i \binom{n-2+ \left\lfloor \frac{i-1}{2} \right\rfloor}{i} Z^i \left( 1 + Z^{-1} U \right)^{\left\lfloor \frac{i+1}{2} \right\rfloor} \\
&= \sum_i \left[ 2 \binom{n-1+ \lfloor \frac{i-1}{2} \rfloor}{i} + \binom{n-2+ \lfloor \frac{i-1}{2} \rfloor}{i-2} - \binom{n-2+\lfloor \frac{i-1}{2} \rfloor}{i} \right]
Z^i \left( 1 + Z^{-1} U \right)^{\left\lfloor \frac{i+1}{2} \right\rfloor} \\
&= \sum_i \binom{n + \lfloor \frac{i-1}{2} \rfloor }{i} Z^i \left( 1 + Z^{-1} U \right)^{\left\lfloor \frac{i+1}{2} \right\rfloor} = f_n.
\end{align*}
In the second line we collect the sums together, shifting $i$ to make a sum over $Z^i (1+Z^{-1} U)^{\lfloor \frac{i+1}{2} \rfloor}$. In the last line we apply the binomial relation $\binom{a}{b} + \binom{a}{b+1} = \binom{a+1}{b+1}$ three times.
\end{proof}

Now substituting (\ref{eqn:Z}) for $Z$ into $r_n$, for $n \geq 0$, gives
\[
\sum_{i=0}^{2n-1} \binom{n + \lfloor \frac{i-1}{2} \rfloor}{i} \left( \frac{(M^2 - 1)(1-L)}{L+M^2} \right)^i \left( 1 + \frac{M+LM^{-1}}{(M-M^{-1})(1-L)} \right)^{\left\lfloor \frac{i+1}{2} \right\rfloor}.
\]
We observe that
\[
1 + \frac{M+LM^{-1}}{(M-M^{-1})(1-L)} = \frac{M^2 - LM^{-2}}{1-L},
\]
and $\lfloor \frac{i}{2} \rfloor + \lfloor \frac{i+1}{2} \rfloor = i$ for all integers $i$. The resulting expression,
\[
\sum_{i=1}^{2n-1} \binom{n + \lfloor \frac{i-1}{2} \rfloor }{i} \left( \frac{M^2 - 1}{M^2 + L} \right)^i \left( 1 - L \right)^{\left\lfloor \frac{i}{2} \right\rfloor} \left( M^2 - LM^{-2} \right)^{\left\lfloor \frac{i+1}{2} \right\rfloor},
\]
once denominators are cleared to give a polynomial, gives the A-polynomial $A_n(L,M)$. As explained in \cite{Hoste-Shanahan04}, we multiply by $M^{2n} (L+M^2)^{2n-1}$. Similarly, for $n \leq 0$, we multiply by $M^{-2n} (L+M^2)^{-2n}$. This gives the desired formula, proving the theorem.

\addcontentsline{toc}{section}{References}

\small

\bibliography{danbib}
\bibliographystyle{amsplain}

\end{document}